\documentclass[12pt]{amsart}
\usepackage[all]{xy}

\usepackage{graphicx}  

\usepackage{amsmath}
\usepackage{amssymb}
\usepackage{amsfonts}
\usepackage{mathrsfs}
\usepackage{mathtools}

\usepackage{color}

   \newcommand{\Cdb}{\mbox{$\mathbb{C}$}}
   \newcommand{\Ddb}{\mbox{$\mathbb{D}$}}

   \newcommand{\Ndb}{\mbox{$\mathbb{N}$}}

   \newcommand{\Rdb}{\mbox{$\mathbb{R}$}}

   \newcommand{\Zdb}{\mbox{$\mathbb{Z}$}}

   \newcommand{\C}{\mbox{${\mathcal C}$}}

   \renewcommand{\P}{\mbox{${\mathcal P}$}}
   
   \newcommand{\R}{\mbox{${\mathcal R}$}}

   \newcommand{\norm}[1]{\Vert#1\Vert}
   \newcommand{\bignorm}[1]{\bigl\Vert#1\bigr\Vert}
   \newcommand{\Bignorm}[1]{\Bigl\Vert#1\Bigr\Vert}
   \newcommand{\biggnorm}[1]{\biggl\Vert#1\biggl\Vert}

\newtheorem{theorem}{Theorem}
\newtheorem{lemma}[theorem]{Lemma}
\newtheorem{proposition}[theorem]{Proposition}
\newtheorem{corollary}[theorem]{Corollary}
\newtheorem{definition}[theorem]{Definition}

\theoremstyle{remark}

\newtheorem{remark}[theorem]{Remark}

\begin{document}

\title{The Ritt property of subordinated operators in the group case}

\author[F. Lancien]{Florence Lancien}
\email{florence.lancien@univ-fcomte.fr}
\author[C. Le Merdy]{Christian Le Merdy}
\email{clemerdy@univ-fcomte.fr}
\address{Laboratoire de Mathématiques de Besan\c con, UMR 6623, 
CNRS, Universit\'e Bourgogne Franche-Comt\'e,
25030 Besan\c{c}on Cedex, FRANCE}

\date{\today}

\maketitle

\begin{abstract}
Let $G$ be a locally compact abelian group, let $\nu$ be a regular 
probability measure on $G$, let $X$ be a Banach space, let 
$\pi\colon G\to B(X)$ be a bounded strongly continuous
representation. Consider the average (or subordinated) operator
$S(\pi,\nu) = \int_{G} \pi(t)\,d\nu(t)\,\colon X\to X$.
We show that if $X$ is a UMD Banach lattice and $\nu$ has bounded angular 
ratio, then $S(\pi,\nu)$ is a Ritt operator with a bounded $H^\infty$ functional calculus.
Next we show that if $\nu$ is the square of a symmetric probability 
measure and $X$ is $K$-convex, then $S(\pi,\nu)$ is a Ritt operator. We further show
that this assertion is false on any non $K$-convex space $X$.
\end{abstract}


\section{Introduction}\label{Intro}
Let $G$ be a locally compact abelian group and let $M(G)$ denote the
Banach algebra of all bounded regular Borel measures on $G$. Let $X$ be
a complex Banach space and let $B(X)$ denote the Banach algebra of all
bounded operators on $X$. Let $\pi\colon G\to B(X)$ be a representation, that is,
$\pi(t+ s)=\pi(t)\pi(s)$ for any $t,s$ in $G$, and $\pi(e)=I_X$ (where $e$ and $I_X$
denote the unit of $G$ and the identity operator on $X$, respectively). Assume further 
that $\pi$ is bounded, that is $\sup_{t\in G}\norm{\pi(t)}\,<\infty\,$, and that $\pi$
is strongly continuous, that is, for any $x\in X$, the mapping 
$t\mapsto\pi(t)x$ is continuous from $G$ into $X$. To any probability measure 
$\nu\in M(G)$, one can associate the average operator
\begin{equation}\label{Def-Average}
S(\pi,\nu) = \int_{G} \pi(t)\,d\nu(t)\,\in B(X),
\end{equation}
where the integral is defined in the strong sense. 

In this paper we are interested in the following two questions.
\begin{itemize}
\item [Q.1] When is $S(\pi,\nu)$ a Ritt operator?
\smallskip
\item [Q.2] When does $S(\pi,\nu)$ admit a bounded $H^\infty$ functional calculus? 
\end{itemize}
Background on Ritt operators and their $H^\infty$ functional calculus 
(along with references) will be given in Section 
\ref{Ritt} below.

Average operators appear in various contexts, notably 
in ergodic theory. When $X$ is a function space, 
the behaviour of the norm of the powers of $S(\pi,\nu)$,
the almost everywhere convergence of these powers and various
maximal and oscillation inequalities were widely studied, see \cite{Der, JR, JRT, LW} 
and the references therein. Recent papers on these topics \cite{CCL, Cun, LX1, LX2} show 
the importance of the Ritt property and $H^\infty$ functional calculus 
in the behaviour of the powers of operators $S(\pi,\nu)$. 
This is the source of motivation for this paper.

Let $U\in B(X)$ be an invertible operator such that 
$\sup_{k\in\footnotesize{\Zdb}}\norm{U^k}\,<\infty$. Then the mapping
$\pi\colon\Zdb\to B(X)$ defined by $\pi(k)=U^k$ for any $k\in\Zdb$
is a representation and any bounded representation of $\Zdb$
on $X$ has this form. A probability measure on $\Zdb$ is given by
a sequence $\nu=(c_k)_{k\in\footnotesize{\Zdb}}$ of nonnegative real
numbers such that $\sum_k c_k =1$. In this case, we have
$$
S(\nu,\pi) = \sum_{k=-\infty}^{\infty} c_k U^k.
$$
Thus $S(\nu,\pi)$ is subordinated to $U$ in the sense of \cite{Dun}.
The special case $G=\Zdb$ therefore indicates that average operators 
(\ref{Def-Average}) may be regarded as subordinated operators 
in the context of group representations. Subordination operators
induced by probability measures on the semigroup $\Ndb$ were
extensively studied recently \cite{BGT, Dun, GT1, GT2}. In this
context a major question is to determine when $\sum_{n=0}^\infty
c_k T^k\,$ is a Ritt operator for a nonnegative sequence 
$(c_k)_{k\geq 0}$ with $\sum_k c_k =1$ and a power bounded $T\colon X\to X$.
Question Q.1 in the present paper should be considered as
its analogue in the group case.

Our results in Sections 3-5 emphasize the role of Banach space
geometry in these issues. In Section \ref{UMD} we recall the so-called 
bounded angular ratio (BAR) condition and extend some of the results
in \cite{CCL}. We show that if $\nu$ 
has BAR and $X$ is a UMD Banach lattice, then 
$S(\pi,\nu)$ is a Ritt operator and it admits a bounded $H^\infty$ 
functional calculus for any $\pi$ as above. 

Section \ref{K} deals with the case when $X$ is a $K$-convex Banach space
and $\nu$ is the square of a symmetric
probability measure. In this case we show (see Theorem \ref{main})
that for any bounded strongly continuous
representation $\pi\colon G\to B(X)$, $S(\pi,\nu)$ is a Ritt operator.
In the case when $\nu=\lambda_p^X$ is the regular representation on $L^p(G;X)$
(see (\ref{Trans}) for the definition),
this result is a discrete analogue of Pisier's 
Theorem \cite{Pis} showing the analyticity of the 
tensor extension of any convolution semigroup associated with a family
of symmetric probability measures.

Section 5 provides examples of pairs $(\pi,\nu)$ for which 
Q.1 (and hence Q.2) has a negative answer. These examples further
show that the $K$-convexity assumption is unavoidable in Theorem \ref{main}. 

\bigskip
We conclude this introduction with a few notations and conventions.
Unless otherwise specified, $G$ denotes an arbitrary locally compact
abelian group, equipped with a fixed Haar measure $dt$. 
For any $1\leq p \leq\infty$, we let $L^p(G)$ 
denote the $L^p$-space associated to this measure. We let 
$\widehat{G}$ denote the dual group of $G$ and, for any
$\nu\in M(G)$, we let $\widehat{\nu}\colon\widehat{G}\to\Cdb\,$ denote
the Fourier transform of $\nu$. For any 
measurable subset $V\subset G$, we let $\vert V\vert$ and 
$\chi_V\colon G\to \Rdb$ denote the Haar
measure of $V$ and the characteristic function of $V$, respectively.

For any $t\in G$, let $\lambda_p(t)\colon L^p(G)\to L^p(G)$ be the 
translation operator defined
by $[\lambda_p(t)f](s)= f(s-t)$ for any $f\in L^p(G)$. We say that an
operator $T\colon L^p(G)\to L^p(G)$ is a Fourier multiplier
if $T\lambda_p(t) = \lambda_p(t) T$ for any $t\in G$.

Let $(\Omega,\mu)$ be a measure space.
For any $1\leq p\leq \infty$ and for any Banach space $X$, we let $L^p(\Omega;X)$ 
be the Bochner space of measurable functions
$f\colon \Omega\to X$ (defined up to almost everywhere zero functions)
such that the norm function $\norm{f(\cdotp)}$
belongs to $L^p(\Omega)$ (see e.g. \cite[Chapter II]{DU}).
For $p\not=\infty$, the algebraic tensor product 
$L^p(\Omega)\otimes X$ is dense in
$L^p(\Omega;X)$.

Let $1\leq p<\infty$, let $T\colon L^p(\Omega_1)\to L^p(\Omega_2)$
and let $S\colon X\to X$ be bounded operators. If $T\otimes S\colon 
L^p(\Omega_1)\otimes X\to L^p(\Omega_2)\otimes X$ extends to a bounded
operator from $L^p(\Omega_1;X)$ into $L^p(\Omega_2;X)$, then we let 
$$
T\overline{\otimes} S\colon L^p(\Omega_1;X)\longrightarrow L^p(\Omega_2;X)
$$
denote this extension. It is well-known that if $T$ is positive (i.e. $T(f)\geq 0$ 
for any $f\geq 0$), then $T\otimes S$ has a bounded
extension for any $S\colon X\to X$.

\section{Ritt operators and their $H^\infty$ functional calculus}\label{Ritt}

An operator $T\colon X\to X$ is called power bounded if
there exists a constant $C_0>0$ such that
$$
\forall\, n\geq 0,\qquad \norm{T^n}\leq C_0.
$$
Then a power bounded $T$ is called a Ritt operator if there 
exists a constant $C_1>0$ such that
\begin{equation}\label{C1}
\forall\, n\geq 1,\qquad n\norm{T^n -T^{n-1}}\leq C_1.
\end{equation}
Ritt operators can be characterized by a spectral condition, as follows.
Let 
$$
\Ddb=\{z\in\Cdb\, :\, \vert z\vert<1\}
$$
be the open unit disc. For
any $T\in B(X)$, let $\sigma(T)$ denote the spectrum of $T$. Then 
$T$ is a Ritt operator if and only if $\sigma(T)\subset \overline{\Ddb}$
and there exists a constant $K>0$ such that 
\begin{equation}\label{Res}
\forall\, z\in\Cdb\setminus\overline{\Ddb},\qquad
\norm{(z -T)^{-1}}\leq \,\frac{K}{\vert z-1\vert}\,.
\end{equation}
This result goes back to \cite{Ly,NZ,Ne}, see also \cite{AL,V} 
for complements.

Let us now turn to functional calculus.  
For any angle $\gamma\in (0,\frac{\pi}{2})$,
consider the so-called Stolz domain $B_\gamma$ as sketched in Figure 1. 
In analytic terms, $B_\gamma$ is defined as the
interior of the convex hull of $1$ and the disc $\{z\in\Cdb\, :\,\vert z\vert<\sin\gamma\}$.

\begin{figure}[ht]
\vspace*{2ex}
\begin{center}
\includegraphics[scale=0.4]{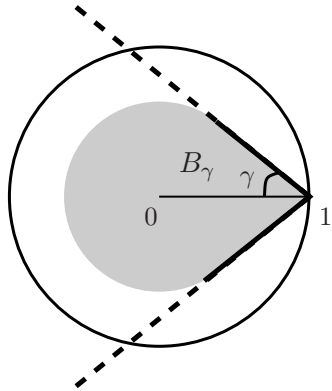}
\begin{picture}(0,0)
\put(-2,65){{\footnotesize $1$}}
\put(-68,65){{\footnotesize $0$}}
\put(-32,81){{\footnotesize $\gamma$}}
\put(-55,85){{\small $B_\gamma$}}
\end{picture}
\end{center}
\caption{\label{f1} Stolz domain}
\end{figure}

It turns out that if $T$ is a Ritt operator, then $\sigma(T)\subset \overline{B_\gamma}$
for some $\gamma\in (0,\frac{\pi}{2})$. In particular we have
$\sigma(T)\subset\Ddb\cup\{1\}$.

One important feature of the above Slolz domains is that for any $\gamma\in (0,\frac{\pi}{2})$
there exists a constant $C_{\gamma}>0$ such that
\begin{equation}\label{Stolz}
\forall\, z\in B_\gamma,\qquad
|1-z|\leq C_\gamma (1-|z|).
\end{equation}

Let $\P$ be the algebra of all complex polynomials. Let $\gamma\in (0,\frac{\pi}{2})$.
Following \cite{L} (to which we refer for more information), 
we say that an operator $T\colon X \to X$ has a bounded $H^\infty(B_\gamma)$ functional calculus
if there exists a constant $C\geq 1$ such that
\begin{equation}\label{FC}
\forall\, \varphi\in\P,\qquad \norm{\varphi(T)}\leq 
C\sup\bigl\{\vert\varphi(z)\vert\, :\, z\in B_\gamma\bigr\}.
\end{equation}
A routine argument 
shows that this condition implies 
$$
\sigma(T)\subset \overline{B_\gamma}.
$$
Further if $T$ satisfies (\ref{FC}), then $T$ is a Ritt operator. 
Indeed (\ref{FC}) applied to $z\mapsto z^n$ 
immediately implies that $\norm{T^n} \leq C$ for any $n\geq 0$. Next define
$\varphi_n\in\P$ by 
$$
\varphi_n(z) = n(z^n - z^{n-1})
$$
for any $n\geq 1$.
Then for any $z\in B_\gamma$, we have
$$
\vert \varphi_n(z) \vert \leq C_\gamma \, n\vert z\vert^{n-1}(1-\vert z\vert)
$$
by (\ref{Stolz}). An elementary computation shows that 
$$
\forall\, t\in (0,1),\qquad nt^{n-1}(1-t)\,\leq\,\Bigl(1-\frac{1}{n}\Bigr)^{n-1}.
$$
We deduce that 
$$
\sup\bigl\{\vert  \varphi_n(z) \vert \, :\, z\in B_\gamma,\ n\geq 1\bigr\}<\infty.
$$
Thus (\ref{FC}) implies the boundedness of the sequence $(\varphi_n(T))_{n\geq 1}$,
hence an estimate (\ref{C1}). We record this simple fact for further use.

\begin{lemma}\label{H-Ritt}
Let $T\colon X\to X$ be an operator satisfying (\ref{FC}) 
for some $\gamma\in (0,\frac{\pi}{2})$.
Then $T$ is Ritt operator.
\end{lemma}

Condition (\ref{FC}) is close but different from the notion of polynomial boundedness. 
Recall that an operator $T\colon X\to X$ is called polynomially bounded 
if there exists a constant $C\geq 1$ such that 
$\norm{\varphi(T)}\leq C\sup\{\vert \varphi(z)\vert \, :\, z\in\Ddb\}$. Clearly if $T$ admits
a bounded $H^\infty(B_\gamma)$ functional calculus, then it is polynomially
bounded. However there exist polynomially bounded Ritt operators 
which do not admit any bounded $H^\infty(B_\gamma)$ functional calculus \cite{LL}.

\begin{definition}\label{H}
We say that a Ritt operator $T\in B(X)$ has a bounded $H^\infty$ functional calculus
if it admits a bounded $H^\infty(B_\gamma)$ functional calculus
for some $\gamma\in (0,\frac{\pi}{2})$.
\end{definition}

Recent papers show the relevance of
this notion for the study of Ritt operators. It is proved in \cite{HH, L} that
for a large class of Banach spaces $X$, a Ritt operator $T\colon X\to X$ 
has a bounded $H^\infty$ functional calculus
in the above sense if and only if $T$ and its adjoint 
$T^*\colon X^*\to X^*$ satisfy certain square functions 
estimates which naturally arise in the harmonic analysis of the 
discrete semigroup $(T_n)_{n\geq 0}$.
We refer the reader to the papers \cite{LX1, LX2} 
for applications of this characterization of 
bounded $H^\infty$ functional calculus.

We now briefly discuss sectorial operators, which will be used in Section \ref{UMD}.
For any angle $\omega\in (0,\pi)$, set 
$$
\Sigma_\omega=\bigl\{
\lambda\in\Cdb^*\, :\,\vert{\rm Arg}(\lambda)\vert<\omega\bigr\}.
$$
Recall that a closed operator $A\colon D(A)\to X$  with dense domain
$D(A)\subset X$ is called sectorial of type $\omega$ if its spectrum
is included in $\overline{\Sigma_\omega}$ and for any $\alpha\in(\omega,\pi)$
there exists a constant $K_\alpha>0$ such that 
\begin{equation}\label{Sector}
\forall\, \lambda\in\Cdb\setminus \overline{\Sigma_\alpha},\qquad
\norm{(\lambda - A)^{-1}}\leq\,\frac{K_\alpha}{\vert\lambda\vert }\,.
\end{equation}
A simple connection between the Ritt condition and sectoriality
is that an operator $T\in B(X)$ is a Ritt operator
if and only if  $A=I_X-T$ is sectorial of
type $<\frac{\pi}{2}$ and $\sigma(T)\subset \Ddb\cup\{1\}$.
This follows from comparing (\ref{Res}) and (\ref{Sector}),
see \cite{NZ} for details.

Let $\R$ be the algebra of all rational functions
with nonpositive degree and poles in the closed
half-line $\Rdb_-$.
For any $\phi\in\R$ and any sectorial operator
$A$, $\phi(A)$ is a well-defined bounded operator on 
$X$. For any $\alpha\in(0,\pi)$, we say that a
sectorial operator $A$ has a bounded $H^\infty(\Sigma_\alpha)$
functional calculus if there exists a constant $C\geq 1$ such that
$$
\forall\, \phi\in\R,\qquad \norm{\phi(A)}\leq C\sup\bigl\{
\vert \phi(\lambda)\vert \; :\, \lambda\in\Sigma_\alpha\bigr\}.
$$
This definition is equivalent to other classical ones that the 
interested reader will find e.g. in \cite{Haase, KW}.

A strong connection between $H^\infty$ calculus for sectorial 
and Ritt operators is given by the following statement.

\begin{theorem}\label{S-R}(\cite[Proposition 4.1]{L})
Let $T\in B(X)$ be a Ritt operator and let $A=I_X-T$. Then $T$
has a bounded $H^\infty$ functional calculus (in the sense of Definition
\ref{H}) if and only if there exists $\alpha\in(0,\frac{\pi}{2})$
such that $A$ has a bounded $H^\infty(\Sigma_\alpha)$
functional calculus.
\end{theorem}

\section{The BAR condition and UMD Banach lattices}\label{UMD}

Let $X$ be an arbitrary Banach space.
We adopt the notation (\ref{Def-Average}) for any $\nu\in M(G)$ and any
bounded strongly continuous representation $\pi\colon G\to B(X)$.

A straightforward application of Fubini's Theorem shows that 
for any $\nu_1,\nu_2\in M(G)$, we have
$$
S(\pi,\nu_2)S(\pi,\nu_1) = S(\pi,\nu_2\ast\nu_1).
$$
This implies that for any $\varphi\in \P$ and for any
$\nu\in M(G)$,
\begin{equation}\label{Pol-nu}
\varphi(S(\pi,\nu))= S(\pi,\nu_\varphi),
\end{equation}
where $\nu_\varphi:=\varphi(\nu)\in M(G)$ is obtained by
applying polynomial functional calculus in the Banach algebra $M(G)$.

For any $1\leq p\leq\infty$, let
$$
C_{\nu,p}^X \colon L^p(G;X)\longrightarrow L^p(G;X)
$$
be the convolution operator defined by setting 
$C_{\nu,p}^X(f) = \nu\ast f =\int_G f(\cdotp -t)\, d\nu(t)\,$ for any
$f$ in $L^p(G;X)$. Note that if $p\not=\infty$ and
$\lambda_p^X\colon G\to
B(L^p(G;X))$ denotes the regular representation defined
by 
\begin{equation}\label{Trans}
\bigl[\lambda_p^X(t)f\bigr](s) = f(s-t),\qquad f\in L^p(G;X),\, s,t\in G,
\end{equation}
then we have
\begin{equation}\label{Conv}
C_{\nu,p}^X = S(\lambda_p^X,\nu).
\end{equation}
For convenience we will write $C_{\nu,p}$
instead of $C_{\nu,p}^X$ when $X=\Cdb$.
Observe that with the notation introduced at the end 
of Section \ref{Intro}, we have
\begin{equation}\label{Tensor}
C_{\nu,p}^X =
C_{\nu,p}\overline{\otimes} I_X
\end{equation} 
for any $1\leq p<\infty$.

We will use the following well-known transference principle, which 
is a variant of \cite[Theorem 2.4]{CW}. 
We include a proof for the sake of completeness.

\begin{proposition}\label{Transfer}
Let $\pi\colon G\to B(X)$ be a bounded strongly continuous
representation and set 
$$
\norm{\pi}=\sup_{t\in G}\norm{\pi(t)}. 
$$
Then for any $1<p<\infty$, we have
\begin{equation}\label{Transfer1}
\bignorm{S(\pi,\nu)\colon X\to X}\leq \norm{\pi}^2
\bignorm{C_{\nu,p}^X\colon L^p(G;X)\to L^p(G;X)}.
\end{equation}
\end{proposition}

\begin{proof}
We first assume that $\nu$ has support in a compact subset $K\subset G$.
Let $V$ be an arbitrary open set, with 
$0<\vert V\vert <\infty$. Let $x\in X$ and define $f\colon G\to X$
by setting
$$
f(t) =\chi_{V-K}(t)\pi(-t)x,\qquad t\in G.
$$
The assumptions imply that $V-K$ has a finite Haar measure, hence $f$ belongs to
$L^p(G;X)$ with 
$$
\norm{f}_p^p \leq\vert V-K\vert \norm{\pi}^p\norm{x}^p.
$$

For any $s\in V$, we may write
$$
S(\pi,\nu)x =\pi(s)\pi(-s)S(\pi,\nu)x = \pi(s) \int_G \pi(t-s)x\,d\nu(t),
$$
which yields
$$
\norm{S(\pi,\nu)x}\leq \norm{\pi} \Bignorm{\int_G \pi(t-s)x\,d\nu(t)}.
$$
Integrating over $V$, we deduce
$$
\vert V\vert\norm{S(\pi,\nu)x}^p\leq \norm{\pi}^p \int_G\chi_V(s)
\Bignorm{\int_G \pi(t-s)x\,d\nu(t)}^p\, ds\,.
$$
Since $\nu$ has support in $K$, $\int_G \pi(t-s)x\,d\nu(t)\,$
is equal to 
$\int_G \chi_{V-K}(s-t) \pi(t-s)x\,d\nu(t)$ for any $s\in V$.
Consequently,
$$
\vert V\vert\norm{S(\pi,\nu)x}^p\leq \norm{\pi}^p \int_G \,
\Bignorm{\int_G \chi_{V-K}(s-t) \pi(t-s)x\,d\nu(t)}^p\, ds\,.
$$
By definition,
$$
\bigl[C_{\nu,p}^X(f)\bigr](s) = \int_G \chi_{V-K}(s-t) \pi(t-s)x\,d\nu(t)
$$
for a.e. $s\in G$. Hence we obtain
\begin{align*}
\vert V\vert\norm{S(\pi,\nu)x}^p & \leq\norm{\pi}^p\norm{C_{\nu,p}^X}^p
\norm{f}_p^p \\ & \leq \norm{C_{\nu,p}^X}^p 
\vert V-K\vert \norm{\pi}^{2p}\norm{x}^p.
\end{align*}
We deduce that
$$
\norm{S(\pi,\nu)x}\leq\biggl(
\frac{\vert V-K\vert }{\vert V\vert}\biggr)^{\frac{1}{p}}
\norm{\pi}^2\norm{C_{\nu,p}}\norm{x}.
$$

The group $G$ is abelian, hence amenable. Thus according to
Folner's condition (see e.g. \cite[Chapter 2]{CW}),
we can choose $V$ such that 
$\frac{\vert V-K\vert }{\vert V\vert}$ is arbitrary close to 1.
Hence the above inequality shows that $\nu$ satisfies (\ref{Transfer1}),
in the case when $\nu$ has compact support.

Now consider an arbitrary $\nu\in M(G)$ (without any assumption on its support).
Since this measure is regular, 
there is a sequence $(K_n)_{n\geq 1}$ of compact 
subsets of $G$ such that 
$$
\vert \nu\vert (G)=\lim_{n\to\infty} \vert \nu\vert (K_n).
$$
Define $\nu_n =\nu_{\vert K_n}$ for any $n\geq 1$. 
Then $\norm{\nu_n - \nu}_{M(G)}\to
0$ hence $\norm{C_{\nu_n,p}^X -  C_{\nu,p}^X}\to 0$ when $n\to\infty$.
Further for any $x\in X$,
$S(\pi,\nu_n)x\to S(\pi,\nu)x$ when $n\to\infty$, by Lebesgue's Theorem.
By the first part of the proof, each $\nu_n$ satisfies (\ref{Transfer1}),
hence $\nu$ satisfies (\ref{Transfer1}) as well.
\end{proof}

Following \cite{CCL} we say that a probability measure $\nu\in M(G)$
has bounded angular ratio (BAR in short) if there exists a constant
$K\geq 1$ such that 
$$
\bigl\vert 1 - \widehat{\nu}(\xi)\bigr\vert\,\leq\, K\bigl( 1- 
\vert \widehat{\nu}(\xi) \vert\bigr),
\qquad \xi\in \widehat{G}.
$$
It is well-known that this holds true if and only if there
exists an angle $\gamma\in (0,\frac{\pi}{2})$ such that
$\widehat{\nu}(\xi)\in \overline{B_\gamma}$ 
for any $\xi\in \widehat{G}$.
The relevance of BAR for the study of Q.1 and Q.2 is shown by 
the following elementary result.

\begin{lemma}\label{Elem}
Let $\nu\in M(G)$ be a probability measure and let $H$
be a Hilbert space. Then $C_{\nu,2}^H$ is a Ritt operator 
if and only if $\nu$ has BAR. In this
case, the operator $C_{\nu,2}^H$ admits a bounded 
$H^\infty$ functional calculus.
\end{lemma}

\begin{proof}
For any $T\in B(L^2(G))$, $T\otimes I_H$ extends to a bounded
operator on $L^2(G;H)$ with $\norm{T\overline{\otimes} I_H}=\norm{T}$.
This immediately implies that $\sigma(C_{\nu,2}^H)=\sigma(C_{\nu,2})$.
Further $C_{\nu,2}^H$ is a normal operator on $L^2(G;H)$, hence
\begin{equation}\label{normal}
\forall\, \varphi\in\P,\qquad
\norm{\varphi(C_{\nu,2}^H)} = \sup\bigl\{\vert\varphi(z)\vert\, :\, 
z\in \sigma(C_{\nu,2})\bigr\}.
\end{equation}
Applying Fourier transform, we have
$$
\sigma(C_{\nu,2})=\overline{\bigl\{\widehat{\nu}(\xi)\, 
:\, \xi\in\widehat{G}\bigr\}},
$$
hence $\nu$ has BAR if and only if there exists 
$\gamma\in (0,\frac{\pi}{2})$ such that
$\sigma(C_{\nu,2})\subset \overline{B_\gamma}$. Combining this
equivalence with (\ref{normal}) yields the result.
\end{proof}

It follows from (\ref{Conv}) and Lemma \ref{Elem} that if a probability 
measure $\nu\in M(G)$ is such that $S(\pi,\nu)$ is a Ritt operator
for all bounded strongly continuous representations $\pi$ acting on
Hilbert space, then 
$\nu$ necessarily has BAR. The next theorem
shows that if we consider representations acting on UMD Banach lattices,
this necessary condition is also sufficient.

We refer the reader to \cite{Bu, HVVW} for background and information on 
the UMD property and to \cite{RF} for a more specific study 
for Banach lattices. We merely recall that any UMD Banach space is reflexive,
that Hilbert spaces and $L^p$-spaces for $1<p<\infty$
are UMD and that the UMD  property is stable under taking 
subspaces and quotients.

\begin{theorem}\label{UMD-BL} 
Let $\nu\in M(G)$ be a probability measure with BAR.
Let $X$ be a UMD Banach lattice. For any bounded strongly continuous
representation $\pi\colon G\to B(X)$, $S(\pi,\nu)$ is a Ritt 
operator with a bounded $H^\infty$ functional calculus.
\end{theorem}

This theorem is close in spirit to \cite{X}. Furthermore
it extends some of the results of \cite{CCL}.
Indeed it is shown in \cite[Proposition 5.2]{CCL} 
that a probability measure $\nu$ has BAR if and only if 
$C_{\nu,p}$ is a Ritt operator
for any $1<p<\infty$. Further \cite[Theorem 5.6]{CCL} says 
that if $X$ is an $L^p$-space for some $1<p<\infty$
and $\nu$ has BAR, then $S(\pi,\nu)$ is a Ritt 
operator for any bounded strongly continuous
representation $\pi\colon G\to B(X)$. $H^\infty$ functional 
calculus is not discussed in \cite{CCL}.

For the proof of Theorem \ref{UMD-BL},
will use complex interpolation, for which we refer to \cite{BL,KMS}. Given any
compatible couple $(X_0,X_1)$ of Banach spaces and any $\theta\in [0,1]$,
we let $[X_0,X_1]_\theta$ denote the interpolation space defined by \cite[Section 4]{KMS}.
The interpolation theorem ensures that if $T\colon X_0+X_1\to X_0+X_1$ is a
linear operator such that $T\colon X_0\to X_0$ and $T\colon X_1\to X_1$
boundedly, then $T\colon [X_0,X_1]_\theta\to [X_0,X_1]_\theta$ boundedly
for any $\theta\in [0,1]$. In the context of Ritt operators, we have the 
following.

\begin{proposition}\label{Blunck} (\cite{B})
Let $(X_0,X_1)$ be a compatible couple of Banach spaces and 
let $T\colon X_0+X_1\to X_0+X_1$ be a linear operator such that 
$T\colon X_0\to X_0$ is power bounded and $T\colon X_1\to X_1$
is a Ritt operator. Then for any $\theta\in (0,1]$,
$T\colon [X_0,X_1]_\theta\to [X_0,X_1]_\theta$ is a Ritt operator.
\end{proposition}

This result was established by Blunck \cite{B} in the case when $X_0,X_1$
are $L^p$-spaces. However the proof works as well in the more general setting
of interpolation couples so we omit it.

\begin{proof}[Proof of Theorem \ref{UMD-BL}]
Let $\nu\in M(G)$ be a probability measure with BAR. 
According to Definition \ref{H} and Lemma \ref{H-Ritt}, 
it suffices to show the existence
of $\gamma\in (0,\frac{\pi}{2})$ and $C\geq 1$ such that
$$
\forall\,\varphi\in\P,\qquad 
\norm{\varphi(S(\pi,\nu))}\leq C\sup
\bigl\{\vert\varphi(z)\vert\, :\, z\in B_\gamma\bigr\}.
$$
By (\ref{Pol-nu}) and Proposition \ref{Transfer}, we have
$$
\norm{\varphi(S(\pi,\nu))}\leq \norm{\pi}^2\norm{\varphi(C_{\nu,2}^X)}
$$
for any $\varphi\in\P$. Hence it suffices to show that 
$T=C_{\nu,2}^X$ satisfies (\ref{FC}) for some $\gamma\in (0,\frac{\pi}{2})$.

Since $X$ is a UMD Banach lattice, it follows from \cite{RF}  that 
there exist a compatible couple $(Y,H)$ and some $\theta\in(0,1]$
such that $H$ is a Hilbert space, 
$Y$ is a UMD Banach space and 
$X=[Y,H]_\theta$ isometrically.
By \cite[Theorem 5.1.2]{BL}, this implies that 
\begin{equation}\label{Inter}
L^2(G;X) =[L^2(G;Y),L^2(G;H)]_\theta\qquad \hbox{isometrically}.
\end{equation}

By  Lemma \ref{Elem}, $C_{\nu,2}^H$ is a Ritt operator.
Further $C_{\nu,2}^Y$ is a contraction hence it follows from
(\ref{Inter}) and Lemma \ref{Blunck} that $C_{\nu,2}^X$ is a Ritt operator.

Now we set 
$$
A^X= I_{L^2(G;X)} - C_{\nu,2}^X
$$
and we similarly define
$A^Y$ on $L^2(G;Y)$ and $A^H$ on $L^2(G;H)$ . 
Then we consider
$$
T_t = e^{-t}e^{t C_{\nu,2}},\qquad t\geq 0.
$$
It is plain that $(T_t)_{t\geq 0}$ is $c_0$-semigroup of contractions
on $L^2(G)$. Its negative generator is equal to  $I_{L^2} - C_{\nu,2}$. 
Moreover $T_t$ is a positive operator (in the lattice sense)
for any $t\geq 0$. Since 
$Y$ is UMD, \cite[Theorem 6]{HP} asserts that the negative 
generator of $\bigl(T_t\overline{\otimes} I_Y\bigr)_{t\geq 0}$
has a bounded $H^\infty(\Sigma_\alpha)$ functional calculus
for any $\alpha\in(\frac{\pi}{2},\pi)$. By construction, this 
negative generator is  $(I_{L^2} - C_{\nu,2})\overline{\otimes} I_Y$,
which is equal to $A^Y$.

On the other hand, $C_{\nu,2}^H$ has a bounded
$H^\infty$ functional calculus by Lemma \ref{Elem}
hence by Theorem \ref{S-R}, there exists $\beta<\frac{\pi}{2}$
such that $A^H$ admits a
bounded $H^\infty(\Sigma_\beta)$ functional calculus.
Applying the interpolation theorem
for $H^\infty$ functional calculus \cite[proposition 4.9]{KKW}, we deduce
that $A^X$ admits a bounded $H^\infty(\Sigma_{\theta\beta +(1-\theta)\alpha})$ 
functional calculus for any $\alpha\in(\frac{\pi}{2},\pi)$. Choosing
$\alpha$ sufficiently close to $\frac{\pi}{2}$ we obtain that 
$A^X$ admits a bounded $H^\infty(\Sigma_{\rho})$ functional calculus
for some $\rho<\frac{\pi}{2}$. By Theorem \ref{S-R} we finally
obtain that $C_{\nu,2}^X$ has a bounded $H^\infty$ functional calculus.
\end{proof}

In the next corollary we focus on
the simple case $G=\Zdb$.

\begin{corollary} Let $(c_k)_{k\in\footnotesize{\Zdb}}$
be a sequence on nonnegative real numbers such that
$\sum_{k=-\infty}^\infty c_k=1$ and there exists a 
constant $K\geq 1$ such that
$$
\Bigl\vert 1 - \sum_{k=-\infty}^\infty c_k e^{ik\theta}
\Bigr\vert\, \leq\,
K\Bigl(1- \Bigl\vert\sum_{k=-\infty}^\infty c_k e^{ik\theta}\Bigr\vert\Bigr),\qquad
\theta\in\Rdb.
$$
\begin{itemize}
\item [(1)] Let $X$ be a UMD Banach lattice and let 
$U\colon X\to X$ be an invertible operator such that 
$$
\sup_{k\in\footnotesize{\Zdb}}\norm{U^k} \, <\infty\,.
$$
Then 
$$
V=\sum_{k=-\infty}^\infty c_k U^k
$$
is a Ritt 
operator with a bounded $H^\infty$ functional calculus.

\item[(2)]
Assume that $c_k =0$ for any $k\leq -1$.
Let $(\Omega,\mu)$ be a measure space, let $1<p<\infty$ and let 
$T\colon L^p(\Omega)\to L^p(\Omega)$ be a positive operator with
$\norm{T}\leq 1$. Then 
$$
S=\sum_{k=0}^\infty c_k T^k
$$
is a Ritt 
operator with a bounded $H^\infty$ functional calculus.
\end{itemize}
\end{corollary}

\begin{proof}
Part (1) is the translation of Theorem \ref{UMD-BL}  in the case $G=\Zdb$.

To prove part (2), we apply the Akcoglu-Sucheston dilation theorem
for positive contractions \cite{AS}: there exist a measure space
$(\Omega',\mu')$, a surjective isometry $U\colon 
L^p(\Omega')\to L^p(\Omega')$ and two contractions
$J\colon L^p(\Omega)\to L^p(\Omega')$ and 
$Q\colon L^p(\Omega')\to L^p(\Omega)$ such that 
$T^k =QU^kJ$ for any $k\geq 0$.

Let $S=\sum_{k\geq 0} c_k T^k$ and $V=\sum_{k\geq 0} c_k U^k$. 
For any integer $n\geq 1$, we have
$$
S^n = \sum_{k\geq 0} c(n)_k T^k
\qquad\hbox{and}\qquad
V^n=\sum_{k\geq 0} c(n)_k U^k,
$$
where $c(n)\in\ell^1$ is defined by $c(1)=(c_k)_{k\geq 0}$
and $c(n)=c(1)\ast\cdots \ast c(1)$ ($n$ times). Then
the above dilation property implies that 
$S^n =QV^nJ$ for any $n\geq 0$.
Thus $\varphi(S)= Q\varphi(V)J$ for any $\varphi\in\P$ and hence
$$
\forall\,\varphi\in \P,\qquad \norm{\varphi(S)}\leq \norm{\varphi(V)}.
$$
The operator $V$ has a bounded $H^\infty$ functional calculus
by part (1). By the above inequality, $S$ also has 
a bounded $H^\infty$ functional calculus.
\end{proof}

\begin{remark} 
Let $\nu\in M(G)$ be a probability measure with BAR and let
$1<p<\infty$. According to Theorem \ref{UMD-BL}, $\C_{\nu,p}$ is a Ritt operator
with a bounded $H^\infty$ functional calculus.

Let $X$ be an $SQ_p$ space, that is, a quotient 
of a subspace of an $L^p$-space. For any $T$ in $B(L^p(G))$,
$T\otimes I_X$ extends to a bounded
operator on $L^p(G;X)$ with $\norm{T\overline{\otimes} I_X}=\norm{T}$.
This implies that 
$$
\forall\, \varphi\in \P, \qquad
\norm{\varphi({C_{\nu,p}^X})}= \norm{\varphi({C_{\nu,p}})}.
$$
The proof of Theorem \ref{UMD-BL} therefore shows that 
for any bounded strongly continuous representation
$\pi\colon G\to B(X)$, $S(\pi,\nu)$ is a Ritt operator
with a bounded $H^\infty$ functional calculus.

It would be interesting to characterize the class of all Banach spaces 
$X$ with this property. 
\end{remark}

\section{Subordination on $K$-convex spaces}\label{K}

Let $\Omega_0$ denote the compact group $\{-1,1\}^{\Ndb}$ equipped
with its normalized Haar measure and for any $n\geq 1$, let 
$\varepsilon_n\colon \Omega_0\to\Rdb$ denote the Rademacher function
defined by setting $\varepsilon_n(\Theta)=\theta_n$ for any
$\Theta=(\theta_i)_{i\geq 1}\in\Omega_0$. Let ${\rm Rad}_2\subset L^2(\Omega_0)$
denote the closed linear span of the $\varepsilon_n$ and 
let $Q\colon L^2(\Omega_0)\to L^2(\Omega_0)$ be the orthogonal projection 
with range equal to ${\rm Rad}_2$.

A Banach space $X$ is called $K$-convex if $Q\otimes I_X$ extends to 
a bounded operator on $L^2(\Omega_0;X)$. In this case, we let 
$$
K_X=\bignorm{Q\overline{\otimes} I_X\colon L^2(\Omega_0;X)\longrightarrow 
L^2(\Omega_0;X)}.
$$
This number is called the $K$-convexity constant of $X$. 
This property plays a fundamental role in
Banach space theory, in relation with the notions of type and cotype. 
In his fundamental paper \cite{Pis}, Pisier
showed that $X$ is $K$-convex if and only if it admits a non trivial Rademacher type.
We refer to \cite{M, Pis, Pis1} for more information on these topics.

$L^p$-spaces are $K$-convex whenever $1<p<\infty$. 
More generally we have the following
well-known fact.

\begin{lemma}\label{Unif}
Let $1<p<\infty$ and let $X$ be a 
$K$-convex Banach space. There exists a constant $C>0$ such that
whenever $(\Omega,\mu)$ is a measure space, the space 
$L^p(\Omega;X)$ is $K$-convex and $K_{L^p(\Omega;X)}\leq C$.
\end{lemma} 

\begin{proof}
It follows from the Khintchine-Kahane inequalities (see e.g. \cite[Theorem 1.e.13]{LT2})
that $Q$ extends to a bounded projection 
$Q_p\colon L^p(\Omega_0)\to L^p(\Omega_0)$, that $X$ is $K$-convex if and only
if $Q_p\otimes I_X$ extends to a bounded operator on $L^p(\Omega_0;X)$, and
that in this case,
$$
A_p K_X\leq 
\bignorm{Q_p\overline{\otimes} I_X\colon L^p(\Omega_0;X)\longrightarrow 
L^p(\Omega_0;X)}\leq B_p K_X
$$
for some universal constants $0<A_p<B_p$.
By Fubini's Theorem, the boundedness of $Q_p\otimes I_X$ on $L^p(\Omega_0;X)$
implies that $Q_p\otimes I_{L^p(\Omega;X)}$ is bounded
on $L^p(\Omega_0;L^p(\Omega;X))$, with
$$ 
\bignorm{Q_p\overline{\otimes} I_{L^p(\Omega;X)}\colon 
L^p(\Omega_0;L^p(\Omega;X))\to L^p(\Omega_0;L^p(\Omega;X))}
= \bignorm{Q_p\overline{\otimes} I_X\colon L^p(\Omega_0;X)\to
L^p(\Omega_0;X)}.
$$
The result follows at once.
\end{proof}

Let $\nu\in M(G)$ be a probability measure. We say that $\nu$
is symmetric provided that $\nu(-V)=\nu(V)$ for any measurable
$V\subset G$. It is clear that $\nu$
is symmetric if and only if $\widehat{\nu}$ is real valued.
In this case, $\widehat{\nu}$ is actually valued in $[-1,1]$.

We will say that such a measure $\nu$ is a square if there exists another 
symmetric probability measure $\eta$ such that
$$
\nu=\eta *\eta.
$$
In this case, $\widehat{\nu} =\widehat{\eta}^2$ is valued in $[0,1]$,
hence $\nu$ has BAR.

Our main result is the following.

\begin{theorem}\label{main}
Let $\nu\in M(G)$ be a symmetric probability measure. Assume that 
$\nu$ is a square and 
let $X$ be a $K$-convex Banach space.
\begin{itemize}
\item [(1)] For any $1<p<\infty$, the convolution operator
$C_{\nu,p}^X = \nu\ast\,\cdotp\,
\colon L^p(G;X)\to L^p(G;X)$ is a Ritt 
operator.
\item [(2)]
For any bounded strongly continuous 
representation $\pi\colon G\to B(X)$, $S(\pi,\nu)$ is a Ritt 
operator.
\end{itemize}
\end{theorem}

If we think of Ritt operators as the discrete analogues of bounded 
analytic semigroups, part (1) of the above theorem should be regarded 
as a discrete version of \cite[Theorem 1.2]{Pis}. 
Its proof will make crucial use of the
following result of Pisier.

\begin{theorem}\label{Pisier}(\cite[Theorem 3.1]{Pis})
Let $Z$ be a $K$-convex Banach space. There exists a constant $C>0$ only depending
on the $K$-convexity constant $K_Z$ such that
for any integer $n\geq 1$ and for any $n$-tuple $(P_1,\ldots,P_n)$ of 
mutually commuting contractive projections on $Z$,
$$
\biggnorm{\sum_{k=1}^n(I_Z - P_k)\prod_{1\leq j\not=k\leq n} P_j}\,\leq C.
$$
\end{theorem}

For any integer $n\geq 2$, the Haar measure on $G^n$ can be defined as 
the $n$-fold product of the Haar measure on $G$. Then 
for any $1\leq p<\infty$, we regard the 
$n$-fold tensor product $L^p(G)\otimes\cdots\otimes L^p(G)$
as a subspace of $L^p(G^n)$ is the usual way. We recall that
$$
L^p(G)\otimes\cdots\otimes L^p(G)\subset L^p(G^n)
$$
is a dense subspace.

\begin{lemma}\label{lem}
Let $1<p<\infty$.
Let $n\geq 2$ and $m\geq 1$ be two integers and let 
$\bigl\{\nu_{ij}\bigr\}_{1\leq i\leq n,\, 
1\leq j\leq m}$ be a family in $M(G)$. For
any $i,j$ as above, let 
$$
S_{ij} = C_{\nu_{ij},p}
= \nu_{ij} \ast \,\cdotp\colon L^p(G)\longrightarrow L^p(G)
$$
be the corresponding convolution operator. Then 
\begin{equation}\label{estimate}
\Bignorm{\sum_{j=1}^m S_{nj}\cdots S_{1j}\overline{\otimes} I_X}_{B(L^p(G;X))}
\,\leq\,\Bignorm{\sum_{j=1}^m S_{nj}\overline{\otimes} 
\cdots \overline{\otimes} S_{1j}\overline{\otimes} I_X}_{B(L^p(G^n;X))}.
\end{equation}
\end{lemma}

\begin{proof}
In the first part of the proof, we assume that 
all the measures $\nu_{ij}$ have a 
compact support. Taking their union, we obtain 
a compact set $K\subset G$ such that 
$\nu_{ij}$ has support included in $K$ for any $1\leq i\leq n,\, 
1\leq j\leq m$. Let $V\subset G$ be an arbitrary open subset 
with $0<\vert V\vert<\infty$. We set $W=V-K$ for convenience.

We fix some $f\in L^p(G;X)$ and introduce $F\colon G^n\to X$ defined by
$$
F(s_1,s_2,\ldots,s_n) = \chi_W(s_2)\cdots\chi_W(s_n) f(s_1+\cdots+s_n).
$$
It is plain that $F$ belongs to $L^p(G^n;X)$, with
$$
\norm{F}_p^p = \vert W\vert^{n-1} \norm{f}_p^p.
$$
We claim that for any $j$, for a.e. $s_1\in G$ and for a.e. $s_2,\ldots,s_n$ in $V$, we have
\begin{equation}\label{Tensorisation}
\bigl[S_{nj}\cdots S_{1j}\overline{\otimes} I_X(f)\bigr](s_1+s_2+\cdots+s_n) 
= \bigl[\bigl(S_{nj}\overline{\otimes} \cdots\overline{\otimes}  
 S_{1j}\overline{\otimes} I_X\bigr)(F)\bigr](s_1,\ldots, s_n).
\end{equation}
Indeed applying $S_{2j}\overline{\otimes} I_X,\ldots, S_{nj}\overline{\otimes} I_X$ successively 
to $S_{1j}\overline{\otimes} I_X (f)$, we have
\begin{align*}
\bigl[S_{nj}\cdots S_{2j} & S_{1j}\overline{\otimes} I_X (f)\bigr](s_1+s_2+\cdots+s_n)\\
 & =\int_{G^{n-1}}  \bigr[S_{1j}(f)\bigr](s_1+s_2+\cdots+s_n - 
t_2-\cdots - t_n) d\nu_{2j}(t_2)\cdots d\nu_{nj}(t_n)
\end{align*}
for a.e. $(s_1,s_2,\ldots,s_n)$ in $G^n$.
It follows from the fact that
$\nu_{2j},\ldots,\nu_{nj}$ have support in $K$
that whenever $s_2,\ldots,s_n$ belong to $V$, 
this is equal to
$$
\int_{G^{n-1}}  \chi_W(s_2 -t_2)\cdots \chi_W(s_n-t_n)   
\bigr[S_{1j}(f)\bigr](s_1+s_2+  \cdots+s_n - t_2-\cdots - t_n)\,
d\nu_{2j}(t_2)\cdots d\nu_{nj}(t_n),
$$
and hence to
\begin{align*}
& \int_{G^{n-1}} \chi_W(s_2 -t_2)\cdots \chi_W(s_n-t_n)
\biggl(\int_G f\Bigl(\sum_{i=1}^n (s_i-t_i)\Bigr)\,
d\nu_{1j}(t_1)\,\biggr)\,
d\nu_{2j}(t_2)\cdots d\nu_{nj}(t_n)\\ 
 = & \int_{G^n} \chi_W(s_2 -t_2)\cdots \chi_W(s_n-t_n)
f\Bigl(\sum_{i=1}^n (s_i-t_i)\Bigr)\,d\nu_{1j}(t_1)
d\nu_{2j}(t_2)\cdots d\nu_{nj}(t_n)\\ 
 = & \int_{G^n} F(s_1-t_1,s_2-t_2,\ldots, s_n-t_n)\,d\nu_{1j}(t_1)
d\nu_{2j}(t_2)\cdots d\nu_{nj}(t_n)\\ 
 = &\bigl[\bigl(S_{nj}\overline{\otimes} \cdots\overline{\otimes}  S_{1j}
\overline{\otimes} I_X\bigr)(F)\bigr](s_1,\ldots, s_n)
\end{align*}
as claimed.

To derive the estimate (\ref{estimate}) from this identity, we note that
by translation invariance, 
$$
\Bignorm{\sum_j S_{nj}\cdots S_{1j}\overline{\otimes} I_X(f)}_p^p =\int_G
\Bignorm{\sum_j \bigl[S_{nj}\cdots S_{1j}\overline{\otimes} I_X(f)
\bigr](s_1 +z)}^p\, d s_1\,.
$$
for any $z\in G$. Hence integrating over $V^{n-1}$,
\begin{align*}
\vert V\vert^{n-1} &
\Bignorm{\sum_j S_{nj}\cdots S_{1j} \overline{\otimes} I_X (f)}_p^p 
\\ & = \int_{V^{n-1}} \int_G
\Bignorm{\sum_j \bigl[S_{nj}\cdots  S_{1j}\overline{\otimes} 
I_X(f)\bigr](s_1+s_2+\cdots+s_n)}^p\, 
ds_1\,\cdotp\,ds_2\cdots ds_n\,.
\end{align*}
According to (\ref{Tensorisation}), this implies that 
\begin{align*}
\vert V\vert^{n-1} 
& \Bignorm{\sum_j S_{nj}\cdots S_{1j} \overline{\otimes} I_X(f)}_p^p \\ 
& =\int_{V^{n-1}}\int_{G} 
\Bignorm{\Bigl[\Bigl(\sum_j S_{nj}\overline{\otimes}\cdots\overline{\otimes} S_{1j}\overline{\otimes} I_X
\Bigr)(F)\Bigr](s_1,\ldots, s_n)}^p\,
ds_1\,\cdotp\,ds_2\cdots ds_n\\
& \leq \Bignorm{\sum_j
S_{nj}\overline{\otimes}\cdots\overline{\otimes} 
S_{1j}\overline{\otimes} I_X}^p\,\norm{F}^p\\
& \leq \Bignorm{\sum_j
S_{nj}\overline{\otimes}\cdots\overline{\otimes} 
S_{1j}\overline{\otimes} I_X}^p\,\vert V-K\vert^{n-1} \norm{f}^p.
\end{align*}
This shows that 
$$
\Bignorm{\sum_j S_{nj}\cdots S_{1j} \overline{\otimes} I_X}
\leq\biggl(\frac{\vert V-K\vert}{\vert V\vert}\biggr)^{\frac{n-1}{p}}\,
\Bignorm{\sum_j
S_{nj}\overline{\otimes}\cdots\overline{\otimes} S_{1j}\overline{\otimes} I_X}.
$$
As indicated in the proof of Proposition \ref{Transfer}, we can choose $V$ such that 
$\frac{\vert V-K\vert}{\vert  V\vert}$ is arbitrary close to 1.
Consequently the above estimate implies (\ref{estimate}) under the assumption that 
all $\nu_{ij}$ have compact support.

The argument at the end of Proposition \ref{Transfer} can be easily adapted to 
deduce the general result from this special case. Details are left to the reader.
\end{proof}

\begin{proof}[Proof of Theorem \ref{main}]
Let $\pi$ be as in part (2) and let $S=S(\pi,\nu)\in B(X)$. 
It follows from (\ref{Pol-nu}) and Proposition \ref{Transfer} 
that for any integer $n\geq 1$,
$$
\bignorm{S^{n} -S^{n-1}}\leq \norm{\pi}^2
\bignorm{(C_{\nu,2}^X)^{n} - (C_{\nu,2}^X)^{n-1}}.
$$
Hence part (2) is a consequence of part (1). 
We now aim at proving part (1).

It is plain that $C_{\nu,p}^X$ is a contraction, hence a power bounded
operator (this does not require the $K$-convexity assumption).

We fix $1<p<\infty$ and set $T=C_{\nu,p}\colon L^p(G)\to L^p(G)$.
By assumption, there exists a symmetric probability
measure $\eta\in M(G)$ such that 
$T$ is the square of $C_{\eta,p}$.

For any $1\leq q\leq \infty$, $C_{\eta,q}=\nu\ast\,\cdotp\,$ is a positive contraction 
on $L^q(G)$. Since $\eta$ is symmetric, the operator
$C_{\eta,2}$ is selfadjoint on the Hilbert space $L^2(G)$. 
Further $C_{\eta,\infty}(1) = 1$, because $\eta$ is a probability measure. 
It therefore follows from 
Rota's dilation Theorem (see e.g. \cite[p. 106]{S}) that there exist a measure 
space $(\Omega,\mu)$, two positive contractions 
$J\colon L^p(G)\to L^p(\Omega)$ and $Q\colon L^p(\Omega)\to L^p(G)$, as well
as a conditional expectation $E\colon L^p(\Omega)\to L^p(\Omega)$ such that
\begin{equation}\label{Rota}
T = Q E J\qquad\hbox{and}\qquad I_{L^p(G)}= QJ.
\end{equation}
In the sequel we will make essential use of the fact that $E$ 
is a positive contraction and $E^2=E$.

Let $n\geq 1$ be an integer. We may write
$$
- n(T^n-T^{n-1}) = nT^{n-1}(I_{L^p} -T)
=\sum_{j=1}^n T^{n-j}(I_{L^p} -T)T^{j-1}.
$$
We apply Lemma \ref{lem} with $m=n$ and the family
$\{S_{ij}\}_{1\leq i,j\leq n}$ of convolution
operators on $L^p(G)$ defined by
$$
S_{jj} =I_{L^p} -T\qquad\hbox{and}\qquad
S_{ij}=T \quad\hbox{if}\ i\not=j.
$$
This yields
$$
n\bignorm{(T^n-T^{n-1})\overline{\otimes} 
 I_X}_{B(L^p(G;X))}\leq\qquad\qquad\qquad\qquad\qquad\qquad\qquad
$$
\begin{equation}\label{Tn}
\qquad\qquad\qquad
\Bignorm{\sum_{j=1}^n \underbrace{T\overline{\otimes}
\cdots \overline{\otimes}T}_{n-j}
\overline{\otimes}(I_{L^p} -T)\overline{\otimes} 
\underbrace{T\overline{\otimes}\cdots 
\overline{\otimes} T}_{j-1}\overline{\otimes} I_X}_{B(L^p(G^n;X))}.
\end{equation}
For convenience we let 
$$
Y= L^p(\Omega)\qquad\hbox{and}\qquad
Z_n = L^p(\Omega^n;X).
$$
According to (\ref{Rota}), we have
\begin{align*}
\underbrace{T \otimes
\cdots \otimes T}_{n-j} &
 \otimes (I_{L^p} -T) \otimes
\underbrace{T \otimes \cdots 
 \otimes T}_{j-1} = \\ &
\Bigl(\underbrace{Q\otimes\cdots\otimes Q}_{n}\Bigr)
\Bigl(\underbrace{E \otimes
\cdots \otimes E}_{n-j}
\otimes(I_{Y} -E)\otimes
\underbrace{E\otimes\cdots 
\otimes E}_{j-1}\Bigr)
\Bigl(\underbrace{J\otimes\cdots\otimes J}_{n}\Bigr)
\end{align*}
for any $j=1,\ldots,n$.
Thus the operator in (\ref{Tn}) is the composition
of the following three operators:

\begin{itemize}
\item [(i)] the operator $J\overline{\otimes}\cdots\overline{\otimes} J \overline{\otimes} I_X$,
which is a well-defined contraction from $L^p(G^n;X)$ into $L^p(\Omega^n;X)$, because $J$ is 
a positive contraction;
\item [(ii)] the operator
$$
\sum_{j=1}^n 
\underbrace{E \overline{\otimes}
\cdots \overline{\otimes}E}_{n-j}
\overline{\otimes}(I_{Y} -E)\overline{\otimes} 
\underbrace{E\overline{\otimes}\cdots 
\overline{\otimes} E}_{j-1}\overline{\otimes} I_X\colon Z_n\longrightarrow Z_n,
$$
which is well-defined because $E$ is positive.
\item [(iii)] the operator $Q\overline{\otimes}\cdots\overline{\otimes} Q \overline{\otimes} I_X$,
which is a well-defined contraction from  $L^p(\Omega^n;X)$ into $L^p(G^n;X)$,
because $Q$ is a positive contraction.
\end{itemize}

We deduce that
$$
n\bignorm{(T^n-T^{n-1})\overline{\otimes} I_X}_{B(L^p(G;X))}\leq
\Bignorm{\sum_{j=1}^n 
\underbrace{E \overline{\otimes}
\cdots \overline{\otimes}E}_{n-j}
\overline{\otimes}(I_{Y} -E)\overline{\otimes} 
\underbrace{E\overline{\otimes}\cdots 
\overline{\otimes} E}_{j-1}\overline{\otimes} I_X}_{B(Z_n)}.
$$
Now for any $j=1,\ldots, n$, define $P_j\colon Z_n\to Z_n$ by
$$
P_j = \underbrace{I_{Y} \overline{\otimes}
\cdots \overline{\otimes}I_{Y} }_{n-j}
\overline{\otimes}E\overline{\otimes} 
\underbrace{I_{Y} \overline{\otimes}\cdots 
\overline{\otimes} I_{Y} }_{j-1}\overline{\otimes} I_X.
$$ 
Then the operator in the right-hand side
of the above inequality is equal to
\begin{equation}\label{Final}
\sum_{k=1}^n(I_{Z_n} - P_k)\prod_{1\leq j\not=k\leq n} P_j.
\end{equation}
Since $E$ is a positive contraction, each
$P_j$ is a contraction. Moreover the $P_j$ 
mutually commute. Further $Z_n$ is $K$-convex 
and $\sup_{n\geq 1}K_{Z_n}\, <\infty$ by Lemma
\ref{Unif}.
It therefore follows from Theorem \ref{Pisier} that the operators in (\ref{Final})
are uniformly bounded. Consequently,
$$
\sup_{n\geq 1}n\,\bignorm{(T^n-T^{n-1})\overline{\otimes}I_X}_{B(L^p(G;X))}\, <\infty.
$$
According to (\ref{Tensor}), this proves 
that $C_{\nu,p}^X = T\overline{\otimes}I_X$ is a Ritt operator.
\end{proof}

We will show in Proposition \ref{density} below
that if a symmetric probability $\nu$ has a density 
we do not need to assume that $\nu$ is a square in the statement
of Theorem \ref{main}.

We identify any $h\in L^1(G)$ with the measure $\nu$
with density $h$ and use the notations
$\C_{h,p}, C_{h,p}^X, S(\pi,h)$ instead of  
$\C_{\nu,p}, C_{\nu,p}^X, S(\pi,\nu)$. With this
convention, $h$ is a probability when $h\in L^1(h)_+$
and $\norm{h}_1=1$ and $h$ is symmetric when $h(t)=h(-t)$
for a.e. $t\in G$. Note that a symmetric probability
$h\in L^1(G)$ has BAR if and only if
there exists $a\in(-1,1]$ such that
$$
\forall\, \xi\in \widehat{G},\qquad
a\leq \widehat{h}(\xi)\leq 1.
$$

A classical result asserts that for any $h\in L^1(G)$, 
we have $\sigma(C_{h,p}) =\sigma(C_{h,2})$ for any 
$1\leq p<\infty$ (see e.g. \cite{Z}).
This implies that 
\begin{equation}\label{spectrum}
\sigma(C_{h,p}^X)
=\sigma(C_{h,2})
\end{equation}
for any $1\leq p<\infty$ and any Banach space $X$. Indeed let $\lambda\in\Cdb\setminus
\sigma(C_{h,1})$. The operator $\lambda I_{L^1} - C_{h,1}$ is a Fourier multiplier hence
its inverse is a Fourier multiplier on $L^1(G)$. 
Consequently there exists $\rho\in M(G)$ such that
$(\lambda I_{L^1} - C_{h,1})^{-1} = C_{\rho,1}=\rho\ast\,\cdotp\colon L^1(G)\to L^1(G)$.
This implies that $\lambda I_{L^p(X)} - C_{h,p}^X$ is invertible, with inverse equal
to $C_{\rho,p}^X$. Hence $\sigma(C_{h,p}^X)$ is included in
$\sigma(C_{h,1})=\sigma(C_{h,2})$. The reverse inclusion is clear.

\begin{proposition}\label{density}
Let $h\in L^1(G)_+$ with $\norm{h}_1=1$. Assume that
$h$ is symmetric. If $h$ has BAR and $X$ is a $K$-convex Banach space,
then $C_{h,p}^X$ is a Ritt operator for any $1<p<\infty$. Furthermore
for any bounded strongly continuous representation $\pi\colon G\to B(X)$,
$S(\pi,h)$ is a Ritt operator.
\end{proposition}

\begin{proof}
As in the proof of Theorem \ref{main} it suffices to
prove the first assertion.

We let $T=C_{h,p}^X$. By the BAR assumption and
Lemma \ref{Elem}, $C_{h, 2}$ is a Ritt operator
hence $-1\notin\sigma(C_{h, 2})$. Applying (\ref{spectrum})
we deduce that $I+T$ is invertible.

By Theorem \ref{main}, $T^2$ is a Ritt operator. Hence
the sequence $\bigl(n(T^{2n}- T^{2(n-1)})\bigr)_{n\geq 1}$
is bounded. Writing
$$
T^{2n}- T^{2(n-1)} = T^{2(n-1)}(T^2 -I) = T^{2(n-1)}(T -I)(T+I),
$$
we deduce that the sequence $\bigl(n T^{2(n-1)}(T-I)\bigr)_{n\geq 1}$
is bounded. This immediately implies that 
the sequence $\bigl(n T^{n}(T-I)\bigr)_{n\geq 1}$ itself is bounded, 
that is, $T$ is a Ritt operator.
\end{proof}

Applying the above result in the case $G=\Zdb$, we 
derive the following.

\begin{corollary} 
Let $(c_k)_{k\in \footnotesize{\Zdb}}$ be a sequence on 
nonnegative real numbers such that
$$
\forall\,\theta\in\Rdb,\qquad
0\leq \sum_{k=-\infty}^\infty c_k e^{ik\theta}\,\leq 1.
$$
Let $X$ be a $K$-convex Banach 
space and let $U\colon X\to X$ be an invertible operator
such that $\sup_{k\in\footnotesize{\Zdb}}\norm{U^k}\,<\infty.$
Then 
$$
S=\sum_{k=-\infty}^\infty c_k U^k
$$
is a Ritt operator.
\end{corollary}

\section{A family of counterexamples}\label{C-Ex}

In Theorem \ref{Pitt} below we provide examples which show two facts.
First BAR does not imply that Q.1. has an affirmative answer on 
general Banach spaces, although Theorem \ref{UMD-BL} says that 
this is the case on UMD Banach lattices. Second, the $K$-convexity 
assumption in Theorem \ref{main} is optimal. 

In the sequel we let $\ell^1_n$ (resp. $\ell^\infty_n$) denote the 
$n$-dimensional $L^1$-space (resp. $L^\infty$-space).

We give some background on the so-called
regular operators, which will be used in the
proof of Theorem \ref{Pitt}. 
Let $(\Omega,\mu)$ be a measure space, let $1<p<\infty$
and let $T\colon L^p(\Omega)\to L^p(\Omega)$ be a bounded
operator. We say that $T$ is regular if there exists
a constant $C\geq 0$ such that 
$$
\Bignorm{\sup_{1\leq k\leq n}\bigl\vert T(f_k)\bigr\vert}_{L^p}
\leq C\,\Bignorm{\sup_{1\leq k\leq n}\bigl\vert f_k\bigr\vert}_{L^p}
$$
for any integer $n\geq 1$ and any $f_1,\ldots,f_n$ in $L^p(\Omega)$.
In this case, we let $\norm{T}_r$ denote the smallest $C$ for 
which this property holds; this is called the regular norm of $T$. We mention
(although we will not use it) that an operator $T$ is regular
if and only if it is a linear combination of positive bounded
operators $L^p(\Omega)\to L^p(\Omega)$.

This definition can be reformulated 
as follows: a bounded operator $T$ on $L^p(\Omega)$ is regular if and only if
the tensor extensions $T\otimes I_{\ell^\infty_n}\colon 
L^p(\Omega;\ell^\infty_n)\to L^p(\Omega;\ell^\infty_n)$
are uniformly bounded. It is not hard to deduce that $T$ is regular
if and only if $T\otimes I_X$ extends to a bounded operator
on $L^p(\Omega;X)$ for any Banach space $X$ and that 
$\norm{T\overline{\otimes} I_X\colon L^p(\Omega;X)\to L^p(\Omega;X)}\leq \norm{T}_r$
in this case. This immediately implies that 
a bounded operator
$T\colon L^p(\Omega)\to L^p(\Omega)$ is regular if and only if its adjoint
$T^*\colon L^{p'}(\Omega)\to  L^{p'}(\Omega)$ is regular, and that
\begin{equation}\label{Dual}
\norm{T^*}_r = \norm{T}_r
\end{equation}
in this case.
(Here $p'=p/(p-1)$ denotes the
conjugate of $p$.) 
We refer to \cite{Pis2, Sc} for details and complements.

We say that a Banach space $X$ contains the $\ell^1_n$ 
uniformly if there exists a constant $C\geq 1$ such that 
for any $n\geq 1$, there exist $x_1,x_2,\ldots,x_n$ in 
$X$ such that
$$
\sum_{i=1}^n\vert\alpha_i\vert\,\leq\Bignorm{
\sum_{i=1}^n \alpha_i x_i}\,\leq C\,\sum_{i=1}^n\vert\alpha_i\vert
$$
for any complex numbers $\alpha_1,\alpha_2,\ldots,\alpha_n$. This
is equivalent to the existence, for any $n\geq 1$,
of an $n$-dimensional subspace of $X$ which is $C$-isomorphic 
to $\ell^1_n$. It is shown in \cite{Pis} that a Banach space
contains the  $\ell^1_n$ uniformly if and only if it is not 
$K$-convex. This leads to the following.

\begin{lemma}\label{Reg}
Let $X$ be a Banach space and assume that $X$ is not $K$-convex.
Let $T\colon L^p(\Omega)\to L^p(\Omega)$ be a bounded operator
such that $T\otimes I_X$ extends to a bounded operator 
$$
T\overline{\otimes} I_X\colon L^p(\Omega;X)\longrightarrow L^p(\Omega;X).
$$
Then $T$ is regular.
\end{lemma}

\begin{proof}
Since $X$ is not $K$-convex, there exist a constant $C\geq 1$ such that 
for any $n\geq 1$, $X$ contains an $n$-dimensional subspace $C$-isomorphic
to $\ell^1_n$. As a consequence of the tensor extension assumption,
we obtain that for any $n\geq 1$,
$$
\bignorm{T\otimes I_{\ell^1_n}\colon 
L^p(\Omega;\ell^1_n)\longrightarrow L^p(\Omega;\ell^1_n)}
\leq C\norm{T\overline{\otimes} X}.
$$
Since $L^p(\Omega;\ell^1_n)^* = L^{p'}(\Omega;\ell^\infty_n)$ 
isometrically, this implies that 
for any $n\geq 1$,
$$
\bignorm{T^*\otimes I_{\ell^\infty_n}\colon 
L^{p'}(\Omega;\ell^\infty_n)\longrightarrow  L^{p'}
(\Omega;\ell^\infty_n)}\leq C\norm{T\overline{\otimes} X}.
$$
This shows that $T^*$ is regular. By (\ref{Dual}), the operator $T$ is regular as well.
\end{proof}

\begin{theorem}\label{Pitt}
Assume that $G$ is a non discrete locally compact abelian group and that
$X$ is a non $K$-convex Banach space. Then there exists a symmetric probability
measure $\eta$ on $G$ such that for any 
$1<p<\infty$, $C^{X}_{\eta^2, p}\colon L^p(G;X)
\to L^p(G;X)$ is not a Ritt operator.
\end{theorem}

\begin{proof}
Let $\delta_e$ denote the Dirac measure at point $e$ (the unit of $G$).
Since $G$ is non discrete, there exists a symmetric probability
measure $\eta$ on $G$ such that 
$$
\tau = \delta_e + \eta^2
$$
is not invertible in the Banach algebra $M(G)$. 
We refer to the proof of \cite[Theorem 5.3.4]{R}
for this fact, which is a generalization of the Wiener-Pitt
Theorem (see in particular the last three 
lines of p. 107 in the latter reference).

We fix some $1<p<\infty$.
By \cite[Proposition 5.2]{CCL} (or by Theorem \ref{UMD-BL}), 
$C_{\eta^2, p}\colon L^p(G)\to L^p(G)$
is a Ritt operator. In particular, $-1\in\sigma(C_{\eta^2, p})$.
Thus the operator 
$$
S= I_{L^p} + C_{\eta^2, p}
$$
is invertible. 

Assume now that $C^{X}_{\eta^2, p}\colon L^p(G;X) \to L^p(G;X)$ is
a Ritt operator. Then similarly, the operator $I_{L^p(X)} + C^X_{\eta^2, p}$ is invertible. 
By (\ref{Tensor}), $I_{L^p(X)} + C^X_{\eta^2, p} =S\overline{\otimes} I_X$ hence the 
invertibility of this operator implies that $S^{-1}\otimes X$ extends to a 
bounded operator
$$
S^{-1}\overline{\otimes}I_X\colon L^p(G;X)\longrightarrow L^p(G;X).
$$
Since $X$ is not $K$-convex, it follows from Lemma \ref{Reg} that 
$S^{-1}$ is regular.

The operator $S^{-1}$ is a Fourier multiplier. Hence 
according to Arendt's description
of regular Fourier multipliers \cite[Proposition 3.3]{A},
there exists a measure $\rho\in M(G)$ such that 
$S^{-1}(f) = \rho\ast f$ for any $f\in L^p(G)$. By construction,
$S(f) = \tau*f$ for any $f\in L^p(G)$. Hence we have
$\rho\ast\tau\ast f=f$ for any $f\in L^p(G)$. This yields 
$\rho\ast\tau=\delta_e$ and contradicts the fact that
$\tau$ is not invertible in the Banach algebra $M(G)$.
\end{proof}

\bigskip
\noindent
{\bf Acknowledgements.} 
The two authors were supported by the French 
``Investissements d'Avenir" program, 
project ISITE-BFC (contract ANR-15-IDEX-03).

\bigskip

\end{document}